\documentclass{article}

\usepackage{amsthm,amsfonts,amssymb,amsmath}

\usepackage{a4wide}

\usepackage{cite}

\newtheorem{theorem}{Theorem}
\newtheorem{corollary}[theorem]{Corollary}
\newtheorem{definition}[theorem]{Definition}
\newtheorem{example}[theorem]{Example}

\newtheorem{remark}[theorem]{Remark}

\begin{document}

\title{Fractional Calculus of Variations in Terms
of a Generalized Fractional Integral with Applications
to Physics\thanks{Part of the first author's Ph.D.,
which is carried out at the University of Aveiro
under the Doctoral Programme
\emph{Mathematics and Applications}
of Universities of Aveiro and Minho.
Submitted 01/Jan/2012; revised 25/Feb/2012; accepted 27/Feb/2012;
for publication in \emph{Abstract and Applied Analysis} 
(http://www.hindawi.com/journals/aaa/).}}

\author{Tatiana Odzijewicz$^{1}$\\
\texttt{tatianao@ua.pt}
\and
Agnieszka B. Malinowska$^{2}$\\
\texttt{a.malinowska@pb.edu.pl}
\and
Delfim F. M. Torres$^{1}$\\
\texttt{delfim@ua.pt}}

\date{$^1$Center for Research and Development in Mathematics and Applications\\
Department of Mathematics, University of Aveiro, 3810-193 Aveiro, Portugal\\[0.3cm]
$^2$Faculty of Computer Science, Bia{\l}ystok University of Technology\\
15-351 Bia\l ystok, Poland}

\maketitle


\begin{abstract}
We study fractional variational problems in terms
of a generalized fractional integral with Lagrangians
depending on classical derivatives, generalized fractional integrals
and derivatives. We obtain necessary
optimality conditions for the basic and isoperimetric problems, as well as
natural boundary conditions for free boundary value problems.
The fractional action--like variational approach (FALVA)
is extended and some applications to Physics discussed.

\bigskip


\bigskip


\end{abstract}


\section{Introduction}

The calculus of variations is a beautiful and useful field of mathematics
that deals with problems of determining extrema (maxima or minima)
of functionals \cite{BorisBookI,BorisBookII,book:Brunt}.
It starts with the simplest problem of finding a function extremizing
(minimizing or maximizing) an integral
\begin{equation*}
\mathcal{J}(y)=\int\limits_a^b F(t,y(t),y'(t))dt
\end{equation*}
subject to boundary conditions $y(a)=y_a$ and $y(b)=y_b$.
In the literature many generalizations of this problem were proposed, including
problems with multiple integrals, functionals containing higher-order derivatives,
and functionals depending on several functions \cite{MyID:130,MyID:203,MalinowskaTorres}.
Of our interest is an extension proposed by Riewe in 1996-1997,
where fractional derivatives (real or complex order)
are introduced in the Lagrangian \cite{CD:Riewe:1996,CD:Riewe:1997}.

During the last decade, fractional problems have increasingly attracted the attention of many
researchers. As mentioned in \cite{isi}, Science Watch of Thomson Reuters identified
the subject as an \emph{Emerging Research Front} area. Fractional derivatives
are nonlocal operators and are historically applied
in the study of nonlocal or time dependent processes \cite{book:Podlubny}.
The first and well established application of fractional calculus in Physics
was in the framework of anomalous diffusion, which is related to features observed
in many physical systems. Here we can mention the report \cite{MK} demonstrating
that fractional equations works as a complementary tool in the description of anomalous
transport processes. Within the fractional approach it is possible to include
external fields in a straightforward manner. As a consequence, in a short period of time the list
of applications expanded. Applications include chaotic dynamics \cite{Zaslavsky},
material sciences \cite{Mainardi}, mechanics of fractal and complex media \cite{Carpinteri,Li},
quantum mechanics \cite{Hilfer,Laskin}, physical kinetics \cite{Edelman},
long-range dissipation \cite{Tarasov3}, long-range interaction
\cite{Tarasov2,Tarasov1}, just to mention a few.
One of the most remarkable applications of fractional calculus appears, however,
in the fractional variational calculus, in the context of
classical mechanics. Riewe \cite{CD:Riewe:1996,CD:Riewe:1997}
shows that a Lagrangian involving fractional
time derivatives leads to an equation of motion with nonconservative forces such
as friction. It is a remarkable result since frictional and nonconservative forces
are beyond the usual macroscopic variational treatment and, consequently,
beyond the most advanced methods of classical mechanics \cite{Lanczos}. Riewe generalizes
the usual variational calculus, by considering Lagrangians that dependent on fractional derivatives,
in order to deal with nonconservative forces. Recently, several
different approaches have been developed to generalize the least action principle and the
Euler--Lagrange equations to include fractional derivatives.
Results include problems depending on Caputo fractional derivatives,
Riemann--Liouville fractional derivatives and others
\cite{Almeida:AML,MyID:182,MyID:152,MyID:179,Cresson,jmp,gastao,mal,comBasia:Frac1,comDorota,MyID:181,MyID:207,Sha}.

A more general unifying perspective to the subject is, however, possible,
by considering fractional operators depending on general kernels
\cite{OmPrakashAgrawal,book:Kiryakova,MyID:226}. In this work we follow such an approach,
developing a generalized fractional calculus of variations. We consider
very general problems, where the classical integrals are substituted by generalized fractional
integrals, and the Lagrangians depend not only on classical derivatives
but also on generalized fractional operators. Problems of the type considered here,
for particular kernels, are important in Physics \cite{Nabulsi}.
Here we obtain general necessary optimality conditions, for several types
of variational problems, which are valid for rather arbitrary operators and kernels.
By choosing particular operators and kernels, one obtains
the recent results available in the literature of Mathematical Physics
\cite{Caldirola,Nabulsi2,Nabulsi3,Nabulsi4,Nabulsi,Herrera}.

The paper is organized as follows. In Section~\ref{sec:prelim} we introduce
the generalized fractional operators and prove some of its basic properties.
Section~\ref{sec:fip} is dedicated to prove integration by parts formulas
for the generalized fractional operators. Such formulas are then used in
later sections to prove necessary optimality conditions
(Theorems~\ref{theorem:ELCaputo} and \ref{theorem:EL2}).
In Sections~\ref{sec:fp}, \ref{sec:fpfb} and \ref{sec:fp:iso}
we study three important classes of generalized variational problems:
we obtain fractional Euler--Lagrange conditions
for the fundamental (Section~\ref{sec:fp}) and generalized isoperimetric problems
(Section~\ref{sec:fp:iso}), as well as fractional natural boundary conditions
for generalized free-boundary value problems (Section~\ref{sec:fpfb}).
Finally, two illustrative examples are discussed in detail in Section~\ref{sec:ex},
while applications to Physics are given in Section~\ref{sec:appl:phys}:
in Section~\ref{sub:sec:harosc} we obtain the damped harmonic oscillator
in quantum mechanics, in Section~\ref{sub:sec:FALVA} we show how results
from FALVA Physics can be obtained. We end with Section~\ref{sec:conc}
of conclusion, pointing out an important direction of future research.


\section{Preliminaries}
\label{sec:prelim}

In this section we present definitions and properties of generalized fractional operators.
As particular cases, by choosing appropriate kernels, these operators are reduced
to standard fractional integrals and fractional derivatives.
Other nonstandard kernels can also be considered as particular cases.
For more on the subject of generalized fractional calculus and applications,
we refer the reader to the book \cite{book:Kiryakova}.
Throughout the text, $\alpha$ denotes a
real number between zero and one.
Following \cite{MyID:209}, we use round brackets for the arguments of functions,
and square brackets for the arguments of operators.
By definition, an operator receives and returns a function.

\begin{definition}[Generalized fractional integral]
The operator $K_P^\alpha$ is given by
\begin{equation*}
K_P^{\alpha}\left[f\right](x)
= K_P^{\alpha}\left[t \mapsto f(t)\right](x)
=p\int\limits_{a}^{x}k_{\alpha}(x,t)f(t)dt
+q\int\limits_{x}^{b}k_{\alpha}(t,x)f(t)dt,
\end{equation*}
where $P=\langle a,x,b,p,q\rangle$ is the \emph{parameter set} ($p$-set for brevity),
$x\in[a,b]$, $p,q$ are real numbers, and $k_{\alpha}(x,t)$
is a kernel which may depend on $\alpha$.
The operator $K_P^\alpha$ is referred as the \emph{operator $K$} ($K$-op for simplicity)
of order $\alpha$ and $p$-set $P$, while $K_P^{\alpha}[f]$ is called the
\emph{operation $K$} (or $K$-opn) of $f$ of order $\alpha$ and p-set $P$.
\end{definition}

Note that if we define
\[
G(x,t):= \left\{ \begin{array}{ll}
p k_\alpha(x,t) & \mbox{if $t < x$},\\
q k_\alpha(t,x) & \mbox{if $t \geq x$},
\end{array} \right.
\]
then the operator $K_P^\alpha$ can be written in the form
\begin{equation*}
K_P^{\alpha}\left[f\right](x)
= K_P^{\alpha}\left[ t \mapsto f(t)\right](x)
=\int_a^b G(x,t) f(t) dt.
\end{equation*}
This is a particular case of one of the oldest and most respectable class
of operators, so called Fredholm operators \cite{book:Helemskii,book:Polyanin}.

\begin{theorem}[\textrm{cf.} Example~6 of \cite{book:Helemskii}]
Let $\alpha\in(0,1)$ and $P=\langle a,x,b,p,q\rangle$. If $k_\alpha$ is a square integrable function
on the square $\Delta=[a,b]\times[a,b]$, then
$K_P^{\alpha}:L_2\left([a,b]\right)\rightarrow L_2\left([a,b]\right)$
is well defined, linear, and bounded operator.
\end{theorem}

\begin{theorem}
\label{theorem:L1}
Let $k_\alpha$ be a difference kernel, \textrm{i.e.},
let $k_\alpha\in L_1\left([a,b]\right)$ with $k_\alpha(x,t)=k_\alpha(x-t)$.
Then, $K_P^{\alpha}:L_1\left([a,b]\right)\rightarrow L_1\left([a,b]\right)$
is a well defined bounded and linear operator.
\end{theorem}

\begin{proof}
Obviously, the operator is linear. Let $\alpha\in(0,1)$,
$P=\langle a,t,b,p,q\rangle$, and $f\in L_1\left([a,b]\right)$. Define
\[
F(\tau,t):=
\left\{
\begin{array}{ll}
\left| p k_\alpha(t-\tau)\right|\cdot \left|f(\tau)\right| & \mbox{if $\tau \leq t$}\\
\left| q k_\alpha(\tau-t)\right|\cdot \left|f(\tau)\right| & \mbox{if $\tau > t$}
\end{array} \right.
\]
for all $(\tau,t)\in\Delta=[a,b]\times[a,b]$.
Since $F$ is measurable on the square $\Delta$, we have
\begin{equation*}
\begin{split}
\int_a^b \left(\int_a^b F(\tau,t)dt\right)d\tau
&=\int_a^b\left[\left|f(\tau)\right|\left(\int_{\tau}^b \left|p k_\alpha(t-\tau)\right|dt
+\int_{a}^{\tau}\left| q k_\alpha(\tau-t)\right|dt\right)\right]d\tau\\
&\leq\int_a^b\left|f(\tau)\right| \bigl| |p| - |q| \bigr| \left\|k_\alpha\right\|d\tau\\
&=\bigl| |p| - |q| \bigr| \cdot \left\|k_\alpha\right\|\cdot\left\|f\right\|.
\end{split}
\end{equation*}
It follows from Fubini's theorem that $F$
is integrable on the square $\Delta$. Moreover,
\begin{equation*}
\begin{split}
\left\|K_P^\alpha[f]\right\|
&=\int_a^b\left|p\int_{a}^{t}k_{\alpha}(t-\tau)f(\tau)d\tau
+q\int_{t}^{b}k_{\alpha}(\tau-t)f(\tau)d\tau\right|dt\\
&\leq\int_a^b\left(|p|\int_{a}^{t}\left|
k_{\alpha}(t-\tau)\right|\cdot\left|f(\tau)\right|d\tau
+|q|\int_{t}^{b}\left|k_{\alpha}(\tau-t)\right|
\cdot\left|f(\tau)\right|d\tau\right)dt\\
&=\int_a^b\left(\int_a^b F(\tau,t)d\tau\right)dt\\
&\leq \bigl| |p| - |q| \bigl| \cdot \left\|k_\alpha\right\|\cdot\left\|f\right\|.
\end{split}
\end{equation*}
Hence, $K_P^{\alpha}:L_1\left([a,b]\right)\rightarrow
L_1\left([a,b]\right)$ and $\left\|K_P^\alpha
\right\|\leq \bigl| |p| - |q| \bigr| \cdot \left\|k_\alpha\right\|$.
\end{proof}

\begin{remark}
The $K$-op reduces to the left and the right Riemann--Liouville
fractional integrals from a suitably chosen kernel
$k_{\alpha}(x,t)$ and $p$-set $P$.
Let $k_{\alpha}(x,t) = k_{\alpha}(x-t)=\frac{1}{\Gamma(\alpha)}(x-t)^{\alpha-1}$:
\begin{itemize}
\item if $P=\langle a,x,b,1,0\rangle$, then
\begin{equation*}
K_{P}^{\alpha}\left[f\right](x)
=\frac{1}{\Gamma(\alpha)}\int\limits_a^x(x-t)^{\alpha-1}f(t)dt
=: {_{a}}\textsl{I}^{\alpha}_{x} \left[f\right](x)
\end{equation*}
is the standard left Riemann--Liouville fractional integral of $f$ of order $\alpha$;
\item if $P=\langle a,x,b,0,1\rangle$, then
\begin{equation*}
K_{P}^{\alpha}\left[f\right](x)=\frac{1}{\Gamma(\alpha)}\int\limits_x^b(t-x)^{\alpha-1}f(t)dt
=: {_{x}}\textsl{I}^{\alpha}_{b} \left[f\right](x)
\end{equation*}
is the standard right Riemann--Liouville fractional integral of $f$ of order $\alpha$.
\end{itemize}
\end{remark}

\begin{corollary}
\label{corollary:Ibounded}
Operators ${_{a}}\textsl{I}^{\alpha}_{x}, {_{x}}\textsl{I}^{\alpha}_{b}:
L_1\left([a,b]\right)\rightarrow L_1\left([a,b]\right)$
are well defined, linear and bounded.
\end{corollary}

The generalized fractional derivatives $A_P^\alpha$
and $B_P^\alpha$ are defined in terms of the
generalized fractional integral $K$-op.

\begin{definition}[Generalized Riemann--Liouville fractional derivative]
\label{def:GRL}
Let $P$ be a given parameter set and $0<\alpha < 1$.
The operator $A_P^\alpha$ is defined by
$A_P^\alpha = D \circ K_P^{1-\alpha}$,
where $D$ denotes the standard derivative operator,
and is referred as the \emph{operator $A$} ($A$-op)
of order $\alpha$ and $p$-set $P$,
while $A_P^\alpha[f]$, for a function $f$ such that
$K_P^{1-\alpha}[f]\in AC\left([a,b]\right)$,
is called the \emph{operation $A$} ($A$-opn)
of $f$ of order $\alpha$ and $p$-set $P$.
\end{definition}

\begin{definition}[Generalized Caputo fractional derivative]
\label{def:GC}
Let $P$ be a given parameter set and $\alpha \in (0,1)$.
The operator $B_P^\alpha$ is defined by
$B_P^\alpha =K_P^{1-\alpha} \circ D$,
where $D$ denotes the standard derivative operator,
and is referred as the \emph{operator $B$} ($B$-op)
of order $\alpha$ and $p$-set $P$,
while $B_P^\alpha[f]$, for a function $f\in AC\left([a,b]\right)$,
is called the \emph{operation $B$} ($B$-opn)
of $f$ of order $\alpha$ and $p$-set $P$.
\end{definition}

\begin{remark}
The standard Riemann--Liouville and Caputo fractional derivatives
are easily obtained from the generalized operators
$A_P^\alpha $ and $B_P^\alpha$, respectively.
Let $k_{1-\alpha}(x,t)=k_{1-\alpha}(x-t)
=\frac{(x-t)^{-\alpha}}{\Gamma(1-\alpha)}$:
\begin{itemize}
\item if $P=\langle a,x,b,1,0\rangle$, then
\begin{equation*}
A_{P}^\alpha\left[f\right](x)=\frac{1}{\Gamma(1-\alpha)}
D\left[\xi \mapsto \int\limits_{a}^{\xi} (\xi-t)^{-\alpha}f(t)dt\right](x)
=: {_{a}}\textsl{D}^{\alpha}_{x}\left[f\right](x)
\end{equation*}
is the standard left Riemann--Liouville fractional derivative
of $f$ of order $\alpha$, while
\begin{equation*}
B_{P}^\alpha \left[f\right](x)=\frac{1}{\Gamma(1-\alpha)}
\int\limits_a^x(x-t)^{-\alpha} D[f](t)dt
=: {^{C}_{a}}\textsl{D}^{\alpha}_{x} \left[f\right](x)
\end{equation*}
is the standard left Caputo fractional derivative of $f$ of order $\alpha$;

\item if $P=\langle a,x,b,0,1\rangle$, then
\begin{equation*}
-A_{P}^\alpha \left[f\right](x)=\frac{-1}{\Gamma(1-\alpha)}
D\left[\xi \mapsto \int\limits_{\xi}^b(t-\xi)^{-\alpha}f(t)dt\right](x)
=: {_{x}}\textsl{D}^{\alpha}_{b} \left[f\right](x)
\end{equation*}
is the standard right Riemann--Liouville
fractional derivative of $f$ of order $\alpha$, while
\begin{equation*}
-B_{P}^\alpha \left[f\right](x) =\frac{-1}{\Gamma(1-\alpha)}
\int\limits_x^b(t-x)^{-\alpha} D[f](t)dt\\
=: {^{C}_{x}}\textsl{D}^{\alpha}_{b} \left[f\right](x)
\end{equation*}
is the standard right Caputo fractional derivative
of $f$ of order $\alpha$.
\end{itemize}
\end{remark}


\section{On generalized fractional integration by parts}
\label{sec:fip}

We now prove integration by parts formulas
for generalized fractional operators.

\begin{theorem}[Fractional integration by parts for the $K$-op]
\label{thm:gfip:Kop}
Let $\alpha \in (0,1)$, $P=\langle a,t,b,p,q\rangle$,
$k_{\alpha}$ be a square-integrable function
on $\Delta=[a,b]\times[a,b]$, and $f,g\in L_2\left([a,b]\right)$.
The generalized fractional integral $K_P^{\alpha}$ satisfies
the integration by parts formula
\begin{equation}
\label{eq:fracIP:K}
\int\limits_a^b g(x)K_P^{\alpha}\left[f\right](x)dx
=\int\limits_a^b f(x)K_{P^*}^{\alpha}\left[g\right](x)dx ,
\end{equation}
where $P^{*}=<a,t,b,q,p>$.
\end{theorem}

\begin{proof}
Define
\[
F(\tau,t):=
\left\{
\begin{array}{ll}
\left|pk_\alpha(t,\tau)\right|
\cdot\left|g(t)\right|\cdot \left|f(\tau)\right|
& \mbox{if $\tau \leq t$}\\
\left|qk_\alpha(\tau,t)\right|
\cdot \left|g(t)\right|\cdot\left|f(\tau)\right|
& \mbox{if $\tau > t$}
\end{array}\right.
\]
for all $(\tau,t)\in\Delta$. Applying Holder's inequality, we obtain
\begin{equation*}
\begin{split}
\int_a^b \left(\int_a^b F(\tau,t)dt\right)d\tau
&=\int_a^b\left[\left|f(\tau)\right|\left(\int_{\tau}^b\left|pk_\alpha(t,\tau)\right|
\cdot\left|g(t)\right|dt+\int_{a}^{\tau}\left|qk_\alpha(\tau,t)\right|
\cdot\left|g(t)\right|dt\right)\right]d\tau\\
&\leq \int_a^b\left[\left|f(\tau)\right|\left(
\int_{a}^b\left|pk_\alpha(t,\tau)\right|\cdot
\left|g(t)\right|dt+\int_{a}^{b}\left|qk_\alpha(\tau,t)\right|
\cdot\left|g(t)\right|dt\right)\right]d\tau\\
&\leq \int_a^b\left\{\left|f(\tau)\right|\left[\left(
\int_a^b\left|pk_\alpha(t,\tau)\right|^2dt\right)^{\frac{1}{2}}\left(
\int_a^b\left|g(t)\right|^2dt\right)^{\frac{1}{2}}\right.\right.\\
&\qquad\qquad \left.\left.+\left(\int_a^b\left|qk_\alpha(\tau,t)\right|^2
dt\right)^{\frac{1}{2}}\left(\int_a^b\left|g(t)\right|^2
dt\right)^{\frac{1}{2}}\right]\right\}d\tau.
\end{split}
\end{equation*}
By Fubini's theorem, functions $k_{\alpha,\tau}(t):=k_{\alpha}(t,\tau)$
and $\hat{k}_{\alpha,\tau}(t):=k_{\alpha}(\tau,t)$
belong to $L_2\left([a,b]\right)$ for almost all $\tau\in[a,b]$. Therefore,
\begin{equation*}
\begin{split}
\int_a^b &\left\{\left|f(\tau)\right|\left[\left(\int_a^b\left|pk_\alpha(t,\tau)\right|^2dt
\right)^{\frac{1}{2}}\left(\int_a^b\left|g(t)\right|^2dt\right)^{\frac{1}{2}}\right.\right.\\
&\qquad\left.\left.+\left(\int_a^b\left|qk_\alpha(\tau,t)\right|^2dt\right)^{\frac{1}{2}}\left(
\int_a^b\left|g(t)\right|^2dt\right)^{\frac{1}{2}}\right]\right\}d\tau\\
&=\left\|g\right\|_2\int_a^b\left[\left|f(\tau)\right|\left(\left\|pk_{\alpha,\tau}\right\|_2
+\left\|q\hat{k}_{\alpha,\tau}\right\|_2\right)\right]d\tau\\
&\leq \left\|g\right\|_2 \left(\int_a^b\left|f(\tau)\right|^2d\tau\right)^{\frac{1}{2}}\left(
\int_a^b\left|\left\|pk_{\alpha,\tau}\right\|_2+\left\|q\hat{k}_{\alpha,\tau}\right\|_2\right|^2
d\tau\right)^{\frac{1}{2}}\\
&\leq\left\|g\right\|_2\cdot\left\|f\right\|_2\left(\left\|pk_\alpha\right\|_2
+\left\|q k_\alpha\right\|_2\right) < \infty.
\end{split}
\end{equation*}
Hence, we can use again Fubini's theorem
to change the order of integration:
\begin{equation*}
\begin{split}
\int\limits_a^b g(t)K_P^{\alpha}[f](t)dt
&= p\int\limits_a^b g(t)dt\int\limits_a^t f(\tau)k_{\alpha}(t,\tau)d\tau
+q\int\limits_a^b g(t)dt\int\limits_t^b f(\tau)k_{\alpha}(\tau,t)d\tau\\
&= p\int\limits_a^b f(\tau)d\tau\int\limits_{\tau}^b g(t)k_{\alpha}(t,\tau)dt
+q\int\limits_a^b f(\tau)d\tau\int\limits_a^{\tau} g(t)k_{\alpha}(\tau,t)dt\\
&=\int\limits_a^b f(\tau)K_{P^*}^{\alpha}[g](\tau)d\tau.
\end{split}
\end{equation*}
\end{proof}

\begin{theorem}
\label{thm:IPL1}
Let $0<\alpha<1$ and $P=\langle a,x,b,p,q\rangle$. If $k_{\alpha}(x,t)=k_{\alpha}(x-t)$,
$k_{\alpha}, f\in L_1\left([a,b]\right)$, and $g\in C\left([a,b]\right)$,
then the operator $K_P^{\alpha}$ satisfies the integration
by parts formula \eqref{eq:fracIP:K}.
\end{theorem}

\begin{proof}
Define
\[
F(t,x) := \left\{
\begin{array}{ll}
\left|pk_\alpha(x-t)\right|\cdot\left|g(x)\right|\cdot \left|f(t)\right| & \mbox{if $t \leq x$}\\
\left|qk_\alpha(t-x)\right|\cdot \left|g(x)\right|\cdot\left|f(t)\right| & \mbox{if $t > x$}
\end{array}
\right.\]
for all $(t,x)\in\Delta=[a,b]\times[a,b]$. Since $g$ is a continuous function on $[a,b]$,
it is bounded on $[a,b]$, i.e., there exists a real number $C>0$ such that
$\left|g(x)\right|\leq C$ for all $x\in [a,b]$. Therefore,
\begin{equation*}
\begin{split}
\int_a^b \left(\int_a^b F(t,x)dt\right)dx
&=\int_a^b\left[\left|f(t)\right|\left(\int_{t}^b\left|pk_\alpha(x-t)\right|
\cdot\left|g(x)\right|dx+\int_{a}^{t}\left|qk_\alpha(t-x)\right|
\cdot\left|g(x)\right|dx\right)\right]dt\\
&\leq \int_a^b\left[\left|f(t)\right|\left(\int_{a}^b\left|pk_\alpha(x-t)\right|
\cdot\left|g(x)\right|dx+\int_{a}^{b}\left|qk_\alpha(t-x)\right|
\cdot\left|g(x)\right|dx\right)\right]dt\\
&\leq C\int_a^b\left[\left|f\left(t\right)\right|\left(\int_{a}^b
\left|pk_\alpha(x-t)\right|dx+\int_{a}^{b}\left|qk_\alpha(t-x)\right|dx\right)\right]dt\\
&=C\bigl|\left|p\right|-\left|q\right|\bigr|\left\|k_\alpha\right\|\left\|f\right\|<\infty.
\end{split}
\end{equation*}
Hence, we can use Fubini's theorem to change the order of integration in iterated integrals.
\end{proof}

\begin{theorem}[Generalized fractional integration by parts]
\label{thm:gfip}
Let $\alpha \in (0,1)$ and $P=\langle a,t,b,p,q\rangle$.
If functions $f,K_{P^*}^{1-\alpha}[g] \in AC([a,b])$, and we are in conditions
to use formula \eqref{eq:fracIP:K} (Theorem~\ref{thm:gfip:Kop} or Theorem~\ref{thm:IPL1}), then
\begin{equation}
\label{eq:fip:2}
\int\limits_a^b g(x) B_{P}^\alpha \left[f\right](x)dx
=\left. f(x) K_{P^*}^{1-\alpha}\left[g\right](x)\right|_a^b
-\int_a^b f(x) A_{P^*}^\alpha\left[g\right](x)dx,
\end{equation}
where $P^*=<a,t,b,q,p>$.
\end{theorem}

\begin{proof}
From Definition~\ref{def:GC} we know that
$B_P^\alpha[f](x) = K_P^{1-\alpha}\left[D[f]\right](x)$.
It follows that
$$
\int_a^b g(x) B_{P}^\alpha[f](x)dx
= \int_a^b g(x) K_{P}^{1-\alpha}\left[D[f]\right](x) dx.
$$
By relation \eqref{eq:fracIP:K}
$$
\int_a^b g(x) B_{P}^\alpha[f](x) dx
= \int_a^b D[f](x) K_{P^*}^{1-\alpha}[g](x) dx,
$$
and the standard integration by parts formula
implies \eqref{eq:fip:2}:
$$
\int_a^b g(x) B_{P}^\alpha[f](x)dx
=\left. f(x) K_{P^*}^{1-\alpha}[g](x) \right|_a^b
-\int_a^b f(x) D\left[K_{P^*}^{1-\alpha}[g]\right](x)dx.
$$
\end{proof}

\begin{corollary}[\textrm{cf.} \cite{book:Klimek}]
Let $0<\alpha<1$. If $f, {_x I_b^{1-\alpha}} \left[g\right] \in AC([a,b])$, then
\begin{equation*}
\int_{a}^{b}  g(x) \, {^C_aD_x^\alpha}\left[f\right](x)dx
=\left.f(x){_x I_b^{1-\alpha}} \left[g\right](x)\right|^{x=b}_{x=a}
+\int_a^b f(x){_x D_b^\alpha}\left[g\right](x)dx.
\end{equation*}
\end{corollary}


\section{The generalized fundamental variational problem}
\label{sec:fp}

By $\partial_{i} F$ we denote the partial
derivative of a function $F$ with respect to its $i$th argument.
We consider the problem of finding a function
$y= t \mapsto y(t)$, $t\in[a,b]$,
that gives an extremum (minimum or maximum)
to the functional
\begin{equation}
\label{eq:1}
\mathcal{J}(y)=K_{P_1}^\alpha\left[t \mapsto
F\left(t,y(t),y'(t),B_{P_2}^\beta \left[y\right](t),
K_{P_3}^\gamma\left[y\right](t)\right)\right](b)
\end{equation}
when subject to the boundary conditions
\begin{equation}
\label{eq:2}
y(a)=y_a, \quad y(b)=y_b,
\end{equation}
where $\alpha,\beta, \gamma\in(0,1)$, $P_1=<a,b,b,1,0>$
and $P_j=<a,t,b,p_j,q_j>$, $j=2,3$. For simplicity of notation we introduce
the operator $\left\{ \cdot \right\}_{P_2, P_3}^{\beta,\gamma}$ defined by
\begin{equation*}
\left\{y\right\}_{P_2, P_3}^{\beta,\gamma}(t)
=\left(t,y(t),y'(t),B_{P_2}^\beta \left[\tau \mapsto y(\tau)\right](t),
K_{P_3}^\gamma \left[\tau \mapsto y(\tau)\right](t)\right).
\end{equation*}
With the new notation one can write \eqref{eq:1} simply as
$\mathcal{J}(y)=K_{P_1}^\alpha\left[F\left\{y\right\}_{P_2, P_3}^{\beta,\gamma}\right](b)$.
The operator $K_{P_1}^\alpha$ has kernel $k_\alpha(x,t)$, and operators $B_{P_2}^\beta$
and $K_{P_3}^\gamma$ have kernels $h_{1-\beta}(t,\tau)$ and $h_\gamma(t,\tau)$, respectively.
In the sequel we assume that:
\begin{enumerate}
\item[(H1)] Lagrangian $F\in C^1\left([a,b]\times\mathbb{R}^4;\mathbb{R}\right)$;

\item[(H2)]  functions
$A_{P_2^*}^\beta\left[\tau \mapsto k_\alpha(b,\tau)\partial_4F\left\{y\right\}_{P_2, P_3}^{\beta,\gamma}(\tau)\right]$,
$K_{P_3^*}^\gamma\left[\tau \mapsto k_\alpha(b,\tau)\partial_5 F\left\{y\right\}_{P_2, P_3}^{\beta,\gamma}(\tau)\right]$,\\
$D\left[t \mapsto \partial_3F\left\{y\right\}_{P_2,P_3}^{\beta,\gamma}(t)k_\alpha(b,t)\right]$ and
$t \mapsto k_\alpha(b,t)\partial_2 F \left\{y\right\}_{P_2, P_3}^{\beta,\gamma}(t)$
are continuous on $(a,b)$;

\item[(H3)] functions $t \mapsto \partial_3F\left\{y\right\}_{P_2, P_3}^{\beta,\gamma}(t)k_\alpha(b,t)$,
$K_{P_2^*}^{1-\beta}\left[\tau \mapsto k_\alpha(b,\tau)
\partial_4 F\left\{y\right\}_{P_2, P_3}^{\beta,\gamma}(\tau)\right] \in AC([a,b])$;

\item[(H4)] kernels $k_\alpha(x,t)$, $h_{1-\beta}(t,\tau)$ and  $h_\gamma(t,\tau)$
are such that we are in conditions to use Theorems~\ref{thm:gfip:Kop},
\ref{thm:IPL1} and \ref{thm:gfip}.
\end{enumerate}

\begin{definition}
A function $y\in C^1\left([a,b];\mathbb{R}\right)$ is said to be
admissible for the fractional variational problem \eqref{eq:1}--\eqref{eq:2},
if functions $B_{P_2}^\beta[y]$ and $K_{P_3}^\gamma[y]$ exist and are continuous on the
interval $[a,b]$, and $y$ satisfies the given boundary conditions \eqref{eq:2}.
\end{definition}

\begin{theorem}
\label{theorem:ELCaputo}
If $y$ is a solution to problem \eqref{eq:1}--\eqref{eq:2},
then $y$ satisfies the generalized Euler--Lagrange equation
\begin{multline}
\label{eq:eqELCaputo}
k_\alpha(b,t)\partial_2 F \left\{y\right\}_{P_2, P_3}^{\beta,\gamma}(t)
-\frac{d}{dt}\left(\partial_3F\left\{y\right\}_{P_2, P_3}^{\beta,\gamma}(t)k_\alpha(b,t)\right)\\
-A_{P_2^*}^\beta\left[\tau \mapsto k_\alpha(b,\tau)\partial_4 F\left\{y\right\}_{P_2, P_3}^{\beta,\gamma}(\tau)\right](t)
+K_{P_3^*}^\gamma\left[\tau \mapsto k_\alpha(b,\tau)\partial_5 F\left\{y\right\}_{P_2, P_3}^{\beta,\gamma}(\tau)\right](t)=0
\end{multline}
for all $t\in(a,b)$.
\end{theorem}

\begin{proof}
Suppose that $y$ is an extremizer of $\mathcal{J}$.
Consider the value of $\mathcal{J}$ at a nearby function
$\hat{y}=y+\varepsilon\eta$,
where $\varepsilon\in\mathbb{R}$ is a small parameter,
and $\eta\in C^1\left([a,b];\mathbb{R}\right)$ is an arbitrary
function with continuous $B$-op and $K$-op.
We require that $\eta(a)=\eta(b)=0$. Let
\begin{equation*}
\begin{split}
\mathcal{J}(\hat{y})&=J(\varepsilon)=K_{P_1}^\alpha\left[t \mapsto
F\left(t,\hat{y}(t),\hat{y}'(t),B_{P_2}^\beta \left[\hat{y}\right](t),
K_{P_3}^\gamma \left[\hat{y}\right](t)\right)\right](b)\\
&=\int\limits_a^b k_\alpha(b,t)F\left(t,y(t)+\varepsilon\eta(t),\frac{d}{dt}\left(y(t)
+\varepsilon\eta(t)\right),B_{P_2}^\beta \left[y+\varepsilon\eta\right](t),
K_{P_3}^\gamma\left[y+\varepsilon\eta\right](t)\right)dt.
\end{split}
\end{equation*}
A necessary condition for $y$ to be an extremizer is given by
\begin{equation}
\label{eq:3}
\begin{split}
\left.\frac{d J}{d\varepsilon}\right|_{\varepsilon=0}=0
&\Leftrightarrow K_{P_1}^\alpha\biggl[
\partial_2 F  \left\{y\right\}_{P_2, P_3}^{\beta,\gamma}  \eta
+\partial_3 F  \left\{y\right\}_{P_2, P_3}^{\beta,\gamma}  D[\eta]\\
&\qquad \qquad + \partial_4 F \left\{y\right\}_{P_2, P_3}^{\beta,\gamma}
B_{P_2}^\beta\left[\eta\right]
+\partial_5 F\left\{y\right\}_{P_2, P_3}^{\beta,\gamma}
K_{P_3}^\gamma \left[\eta\right]\biggr](b)=0\\
&\Leftrightarrow
\int\limits_a^b
\Biggl(\partial_2 F  \left\{y\right\}_{P_2, P_3}^{\beta,\gamma}(t)\eta(t)
+\partial_3 F  \left\{y\right\}_{P_2, P_3}^{\beta,\gamma}(t)\frac{d}{dt}\eta(t)\\
&\qquad + \partial_4 F \left\{y\right\}_{P_2, P_3}^{\beta,\gamma}(t)
B_{P_2}^\beta\left[\eta\right](t)
+\partial_5 F\left\{y\right\}_{P_2, P_3}^{\beta,\gamma}(t)
K_{P_3}^\gamma \left[\eta\right](t)\Biggr)k_\alpha(b,t)dt = 0.
\end{split}
\end{equation}
Using classical and generalized fractional integration by parts formulas
(Theorems~\ref{thm:gfip:Kop}, \ref{thm:IPL1} and \ref{thm:gfip}),
\begin{multline*}
\int_a^b\partial_3F\left\{y\right\}_{P_2, P_3}^{\beta,\gamma}(t)k_\alpha(b,t)
\frac{d}{dt} \eta(t) dt\\
=\left.\partial_3F\left\{y\right\}_{P_2, P_3}^{\beta,\gamma}(t)
k_\alpha(b,t)\eta(t)\right|_a^b
-\int_a^b \frac{d}{dt}\left(\partial_3F\left\{y\right\}_{P_2,
P_3}^{\beta,\gamma}(t)k_\alpha(b,t)\right) \eta(t) dt,
\end{multline*}
\begin{multline*}
\int\limits_a^b\partial_4 F\left\{y\right\}_{P_2, P_3}^{\beta,\gamma}(t)k_\alpha(b,t)
B_{P_2}^\beta\left[\eta\right](t)dt\\
= \left. K_{P_2^*}^{1-\beta}\left[\tau \mapsto k_\alpha(b,\tau)
\partial_4 F\left\{y\right\}_{P_2, P_3}^{\beta,\gamma}(\tau)\right](t) \eta(t) \right|_a^b
-\int\limits_a^b A_{P_2^*}^\beta\left[\tau \mapsto k_\alpha(b,\tau)
\partial_4 F\left\{y\right\}_{P_2, P_3}^{\beta,\gamma}(\tau)\right](t) \, \eta(t) dt
\end{multline*}
and
\begin{equation*}
\int\limits_a^b k_\alpha(b,t)\partial_5 F\left\{y\right\}_{P_2, P_3}^{\beta,\gamma}(t)
K_{P_3}^\gamma\left[\eta\right](t) dt
=\int\limits_a^b K_{P_3^*}^\gamma\left[\tau \mapsto k_\alpha(b,\tau)
\partial_5 F\left\{y\right\}_{P_2, P_3}^{\beta,\gamma}(\tau)\right](t) \, \eta(t) dt,
\end{equation*}
where $P_j^*=<a,t,b,q_j,p_j>$, $j=2,3$.
Because $\eta(a)=\eta(b)=0$, \eqref{eq:3} simplifies to
\begin{multline*}
\int_a^b \Biggr\{k_\alpha(b,t)\partial_2 F \left\{y\right\}_{P_2, P_3}^{\beta,\gamma}(t)
-\frac{d}{dt}\left(\partial_3F\left\{y\right\}_{P_2, P_3}^{\beta,\gamma}(t)k_\alpha(b,t)\right)\\
-A_{P_2^*}^\beta\left[\tau \mapsto k_\alpha(b,\tau)\partial_4 F\left\{y\right\}_{P_2, P_3}^{\beta,\gamma}(\tau)\right](t)
+K_{P_3^*}^\gamma\left[\tau \mapsto k_\alpha(b,\tau)\partial_5 F\left\{y\right\}_{P_2,P_3}^{\beta,\gamma}(\tau)\right](t)\Biggr\}
\eta(t) dt=0.
\end{multline*}
We obtain \eqref{eq:eqELCaputo} by application of the fundamental lemma
of the calculus of variations (see, \textrm{e.g.}, \cite[Section~2.2]{G:H}).
\end{proof}

The next corollary gives an extension of the main result of \cite{jmp}.

\begin{corollary}
\label{corollary:fund1}
If $y$ is a solution to the problem of minimizing or maximizing
\begin{equation}
\label{eq:4}
\mathcal{J}(y)={_{a}}\textsl{I}^{\alpha}_{b}\left[t \mapsto
F\left(t,y(t),y'(t), \, {^C_aD_t^\beta} \left[y\right](t)\right)\right](b)
\end{equation}
in the class $y\in C^1\left([a,b];\mathbb{R}\right)$ subject to the boundary conditions
\begin{equation}
\label{eq:7}
y(a)=y_a, \quad y(b)=y_b,
\end{equation}
where $\alpha,\beta\in(0,1)$, $F\in C^1\left([a,b]\times \mathbb{R}^3;\mathbb{R}\right)$
and $\tau \mapsto (b-\tau)^{\alpha-1}
\partial_4 F\left(\tau,y(\tau),y'(\tau), \, {^C_aD_\tau^\beta} \left[y\right](\tau)\right)$
has continuous Riemann-Liouville fractional derivative $_{t}D^{\beta}_{b}$, then
\begin{multline}
\label{eq:5}
\partial_2 F\left(t,y(t),y'(t), \, {^C_aD_t^\beta} \left[y\right](t)\right)\cdot(b-t)^{\alpha-1}
-\frac{d}{dt}\left\{\partial_3 F\left(t,y(t),y'(t),
\, {^C_aD_t^\beta} \left[y\right](t)\right) \cdot (b-t)^{\alpha-1}\right\}\\
+{_{t}D^{\beta}_{b}}\left[\tau \mapsto (b-\tau)^{\alpha-1}
\partial_4 F\left(\tau,y(\tau),y'(\tau), \, {^C_aD_\tau^\beta} \left[y\right](\tau)\right)\right](t)=0
\end{multline}
for all $t\in(a,b)$.
\end{corollary}

\begin{proof}
Choose $k_\alpha(x,t)=\frac{1}{\Gamma(\alpha)}(x-t)^{\alpha-1}$,
$h_{1-\beta}(t,\tau)=\frac{1}{\Gamma(1-\beta)}(t-\tau)^{-\beta}$,
and $P_2=<a,t,b,1,0>$. Then the $K$-op, the $A$-op and the $B$-op reduce
to the left fractional integral, the left Riemann-Liouville
and the left Caputo fractional derivatives, respectively.
Therefore, problem \eqref{eq:4}--\eqref{eq:7} is a particular case
of problem \eqref{eq:1}--\eqref{eq:2} and \eqref{eq:5} follows
from \eqref{eq:eqELCaputo} with $\partial_5 F=0$.
\end{proof}

The following result is the Caputo analogous to the main result of \cite{DerInt}
done for the Riemann--Liouville fractional derivative.

\begin{corollary}
\label{corollary:fund2}
Let $\beta, \gamma\in (0,1)$. If $y$ is a solution to the problem
\begin{gather*}
\int\limits_a^b F\left(t,y(t),y'(t), {_{a}^{C}}\textsl{D}_t^\beta [y](t),
{_{a}}\textsl{I}_t^\gamma [y](t)\right)dt \longrightarrow \textrm{extr}\\
y \in C^1\left([a,b]; \mathbb{R}\right)\\
y(a)=y_a, \quad y(b)=y_b,
\end{gather*}
then
\begin{multline}
\label{eq:9}
\partial_2 F\left(t,y(t),y'(t), {{_{a}^{C}}\textsl{D}}_t^\beta[y](t),
{_{a}}\textsl{I}_t^\gamma[y](t)\right)
- \frac{d}{dt}\partial_3 F\left(t,y(t),y'(t),
{_{a}^{C}}\textsl{D}_t^\beta[y](t),{_{a}}\textsl{I}_t^\gamma[y](t)\right)\\
+ {_{t}}\textsl{D}_b^\beta \left[\tau \mapsto \partial_4 F\left(\tau,y(\tau),y'(\tau),
{_{a}^{C}}\textsl{D}_\tau^\beta[y](\tau),{_{a}}\textsl{I}_\tau^\gamma[y](\tau)\right)\right](t)\\
+{_{t}}\textsl{I}_b^\beta \left[\tau \mapsto \partial_5 F\left(\tau,y(\tau),y'(\tau),
{_{a}^{C}}\textsl{D}_\tau^\beta[y](\tau),{_{a}}\textsl{I}_\tau^\gamma[y](\tau)\right)\right](t) = 0
\end{multline}
holds for all $t \in [a,b]$.
\end{corollary}

\begin{proof}
The Euler--Lagrange equation \eqref{eq:9}
follows from \eqref{eq:eqELCaputo} by choosing
$p$-sets $P_1=<a,b,b,1,0>$, $P_2=P_3=<a,t,b,1,0>$, and kernels $k_\alpha(x,t)=1$,
$h_{1-\beta}(t,\tau)=\frac{1}{\Gamma(1-\beta)}(t-\tau)^{-\beta}$,
and $h_{\gamma}(t,\tau)=\frac{1}{\Gamma(\gamma)}(t-\tau)^{\gamma-1}$.
\end{proof}

\begin{remark}
In the particular case when the Lagrangian $F$ of Corollary~\ref{corollary:fund2}
does not depend on the fractional integral and the classical derivative, one obtains
from \eqref{eq:9} the Euler--Lagrange equation of \cite{fred:tor}.
\end{remark}


\section{Generalized free-boundary variational problems}
\label{sec:fpfb}

Assume now that in problem \eqref{eq:1}--\eqref{eq:2}
the boundary conditions \eqref{eq:2} are substituted by
\begin{equation}
\label{eq:Free1}
y(a) \textnormal{ is free } \textnormal{ and } y(b)=y_b.
\end{equation}

\begin{theorem}
\label{theorem:NatBound}
If $y$ is a solution to the problem of extremizing functional \eqref{eq:1}
with \eqref{eq:Free1} as boundary conditions, then $y$ satisfies the
Euler--Lagrange equation \eqref{eq:eqELCaputo}. Moreover,
the extra natural boundary condition
\begin{equation}
\label{eq:NatBoundCond}
\partial_3 F\left\{y\right\}_{P_2, P_3}^{\beta,\gamma}(a)k_\alpha(b,a)
+K_{P_2^*}^{1-\beta}\left[\tau \mapsto \partial_4 F\left\{y\right\}_{P_2,
P_3}^{\beta,\gamma}(\tau)k_\alpha(b,\tau)\right](a)=0
\end{equation}
holds.
\end{theorem}

\begin{proof}
Under the boundary conditions \eqref{eq:Free1},
we do not require $\eta$ in the proof of Theorem~\ref{theorem:ELCaputo}
to vanish at $t=a$. Therefore, following the proof
of Theorem~\ref{theorem:ELCaputo}, we obtain
\begin{equation}
\label{eq:8}
\begin{split}
\partial_3 F&\left\{y\right\}_{P_2, P_3}^{\beta,\gamma}(a)k_\alpha(b,a)\eta(a)
+\eta(a)K_{P_2^*}^{1-\beta}\left[\tau \mapsto
\partial_4 F\left\{y\right\}_{P_2, P_3}^{\beta,\gamma}(\tau)k_\alpha(b,\tau)\right](a)\\
&+\int_a^b\eta(t)\Biggl(\partial_2 F \left\{y\right\}_{P_2, P_3}^{\beta,\gamma}(t)k_\alpha(b,t)
-\frac{d}{dt}\left(\partial_3 F\left\{y\right\}_{P_2, P_3}^{\beta,\gamma}(t)k_\alpha(b,t)\right)\\
&-A_{P_2^*}^\beta\left[\tau \mapsto \partial_4 F\left\{y\right\}_{P_2, P_3}^{\beta,\gamma}(\tau)k_\alpha(b,\tau)\right](t)
+K_{P_3^*}^\gamma\left[\tau \mapsto
\partial_5 F\left\{y\right\}_{P_2, P_3}^{\beta,\gamma}(\tau)k_\alpha(b,\tau)\right](t)\Biggr)dt=0
\end{split}
\end{equation}
for every admissible $\eta\in C^1([a,b];\mathbb{R})$ with $\eta(b)=0$.
In particular, condition \eqref{eq:8} holds for those $\eta$ that fulfill $\eta(a)=0$.
Hence, by the fundamental lemma of the calculus of variations, equation \eqref{eq:eqELCaputo}
is satisfied. Now, let us return to \eqref{eq:8} and let $\eta$ again be arbitrary at point $t=a$.
Inserting \eqref{eq:eqELCaputo}, we obtain the natural boundary condition \eqref{eq:NatBoundCond}.
\end{proof}

\begin{corollary}
Let $\mathcal{J}$ be the functional given by
\begin{equation*}
\mathcal{J}(y)={_{a}}\textsl{I}_b^\alpha\left[t \mapsto
F\left(t,y(t),{_{a}^{C}}\textsl{D}_t^\beta[y](t)\right)\right](b).
\end{equation*}
Let $y$ be a minimizer of $\mathcal{J}$ satisfying the boundary
condition $y(b)=y_b$. Then, $y$ satisfies the Euler--Lagrange equation
\begin{equation}
\label{eq:10}
(b-t)^{\alpha-1}\partial_2 F\left(t,y(t),{_{a}^{C}}\textsl{D}_t^\alpha [y](t)\right)
+{_{t}}\textsl{D}_b^\alpha\left[\tau \mapsto (b-\tau)^{\alpha-1}
\partial_3 F\left(\tau,y(\tau),{_{a}^{C}}\textsl{D}_\tau^\beta y(\tau)\right)\right](t)=0
\end{equation}
and the natural boundary condition
\begin{equation}
\label{eq:11}
{_{a}}\textsl{I}_b^{1-\beta}\left[\tau \mapsto (b-\tau)^{\alpha-1}
\partial_3 F\left(\tau,y(\tau),{_{a}^{C}}\textsl{D}_\tau^\beta y(\tau)\right)\right](a)=0.
\end{equation}
\end{corollary}

\begin{proof}
Let functional \eqref{eq:1} be such that
it does not depend on the classical (integer) derivative
$y'(t)$ and on the $K$-op. If $P_2=<a,t,b,1,0>$,
$h_{1-\beta}(t-\tau)=\frac{1}{\Gamma(1-\beta)}(t-\tau)^{-\beta}$,
and $k_{\alpha}(x-t)=\frac{1}{\Gamma(\alpha)}(x-t)^{\alpha-1}$,
then the $B$-op reduces to the left fractional Caputo derivative and
we deduce \eqref{eq:10} and \eqref{eq:11}
from \eqref{eq:eqELCaputo} and \eqref{eq:NatBoundCond}, respectively.
\end{proof}

\begin{corollary}
Let $\mathcal{J}$ be the functional given by
\begin{equation*}
\mathcal{J}(y)=\int\limits_a^b
F\left(t,y(t),y'(t),B_{P_2}^\beta[y](t),K_{P_3}^\gamma[y](t)\right)dt.
\end{equation*}
If $y$ is a minimizer to $\mathcal{J}$ satisfying the boundary
condition $y(b)=y_b$, then $y$ satisfies the Euler--Lagrange equation
\begin{equation}
\label{eq:eqELCaputoCor}
\partial_2 F \left\{y\right\}_{P_2, P_3}^{\beta,\gamma}(t)
-\frac{d}{dt}\partial_3 F\left\{y\right\}_{P_2, P_3}^{\beta,\gamma}(t)
-A_{P_2^*}^\beta\left[\partial_4 F\left\{y\right\}_{P_2, P_3}^{\beta,\gamma}\right](t)
+K_{P_3^*}^\gamma\left[\partial_5 F\left\{y\right\}_{P_2, P_3}^{\beta,\gamma}\right](t)=0
\end{equation}
and the natural boundary condition
\begin{equation}
\label{eq:NatBoundCond2}
\partial_3 F\left\{y\right\}_{P_2, P_3}^{\beta,\gamma}(a)
+K_{P_2^*}^{1-\beta}\left[\partial_4
F\left\{y\right\}_{P_2, P_3}^{\beta,\gamma}\right](a)=0.
\end{equation}
\end{corollary}

\begin{proof}
Choose, in the problem defined by \eqref{eq:1} and \eqref{eq:Free1}, $k_\alpha(x,t) \equiv 1$.
Then, equations \eqref{eq:eqELCaputoCor} and \eqref{eq:NatBoundCond2} follow
from \eqref{eq:eqELCaputo} and \eqref{eq:NatBoundCond}, respectively.
\end{proof}


\section{Generalized isoperimetric problems}
\label{sec:fp:iso}

Let $\xi\in\mathbb{R}$. Among all functions
$y:[a,b]\rightarrow\mathbb{R}$ satisfying boundary conditions
\begin{equation}
\label{eq:IsoBound}
y(a)=y_a, \quad y(b)=y_b,
\end{equation}
and an isoperimetric constraint of the form
\begin{equation}
\label{eq:IsoConstr}
\mathcal{I}\left(y\right)=K_{P_1}^\alpha\left[G\left\{y\right\}_{P_2, P_3}^{\beta,\gamma}\right](b)=\xi,
\end{equation}
we look for the one that extremizes (\textrm{i.e.}, minimizes or maximizes) a functional
\begin{equation}
\label{eq:IsoFunct}
\mathcal{J}\left(y\right)
=K_{P_1}^\alpha\left[F\left\{y\right\}_{P_2, P_3}^{\beta,\gamma}\right](b).
\end{equation}
Operators $K_{P_1}^\alpha$, $B_{P_2}^\beta$ and $K_{P_3}^\gamma$,
as well as function $F$, are the same as in problem \eqref{eq:1}--\eqref{eq:2}.
Moreover, we assume that functional \eqref{eq:IsoConstr} satisfies hypotheses (H1)--(H4).

\begin{definition}
A function $y : [a,b]\to\mathbb R$
is said to be \emph{admissible} for problem
\eqref{eq:IsoBound}--\eqref{eq:IsoFunct} if functions $B_{P_2}^\beta[y]$
and $K_{P_3}^\gamma[y]$ exist and are continuous on $[a,b]$,
and $y$ satisfies the given boundary conditions \eqref{eq:IsoBound}
and the given isoperimetric constraint \eqref{eq:IsoConstr}.
\end{definition}

\begin{definition}
An admissible function $y\in C^1\left([a,b],\mathbb{R}\right)$ is said to be an \emph{extremal}
for $\mathcal{I}$ if it satisfies the Euler--Lagrange equation
\eqref{eq:eqELCaputo} associated with functional in \eqref{eq:IsoConstr}, \textrm{i.e.},
\begin{multline*}
k_\alpha(b,t)\partial_2 G \left\{y\right\}_{P_2, P_3}^{\beta,\gamma}(t)
-\frac{d}{dt}\left(\partial_3G\left\{y\right\}_{P_2, P_3}^{\beta,\gamma}(t)k_\alpha(b,t)\right)\\
-A_{P_2^*}^\beta\left[\tau \mapsto k_\alpha(b,\tau)\partial_4 G\left\{y\right\}_{P_2, P_3}^{\beta,\gamma}(\tau)\right](t)
+K_{P_3^*}^\gamma\left[\tau \mapsto k_\alpha(b,\tau)\partial_5 G\left\{y\right\}_{P_2, P_3}^{\beta,\gamma}(\tau)\right](t)=0,
\end{multline*}
where $P_j^*=<a,t,b,q_j,p_j>$, $j=2,3$, and $t\in(a,b)$.
\end{definition}

\begin{theorem}
\label{theorem:EL2}
If $y$ is a solution to the isoperimetric problem
\eqref{eq:IsoBound}--\eqref{eq:IsoFunct} and
is not an extremal for $\mathcal{I}$, then
there exists a real constant $\lambda$ such that
\begin{multline}
\label{eq:eqEL2}
k_\alpha(b,t)\partial_2 H \left\{y\right\}_{P_2, P_3}^{\beta,\gamma}(t)
-\frac{d}{dt}\left(\partial_3 H\left\{y\right\}_{P_2, P_3}^{\beta,\gamma}(t)k_\alpha(b,t)\right)\\
-A_{P_2^*}^\beta\left[\tau \mapsto k_\alpha(b,\tau)\partial_4 H\left\{y\right\}_{P_2, P_3}^{\beta,\gamma}(\tau)\right](t)
+K_{P_3^*}^\gamma\left[\tau \mapsto k_\alpha(b,\tau)\partial_5 H\left\{y\right\}_{P_2, P_3}^{\beta,\gamma}(\tau)\right](t)=0
\end{multline}
for all $t\in(a,b)$, where $H(t,y,u,v,w)=F(t,y,u,v,w)-\lambda G(t,y,u,v,w)$ and $P_j^*=<a,t,b,q_j,p_j>$.
\end{theorem}

\begin{proof}
Consider a two-parameter family of the form
$\hat{y}=y+\varepsilon_1\eta_1+\varepsilon_2\eta_2$,
where for each $i\in\{1,2\}$ we have $\eta_i(a)=\eta_i(b)=0$.
First we show that we can select $\varepsilon_2\eta_2$ such that
$\hat{y}$ satisfies \eqref{eq:IsoConstr}. Consider the quantity
$\mathcal{I}(\hat{y})=K_{P_1}^\alpha\left[G\left\{\hat{y}\right\}_{P_2,
P_3}^{\beta,\gamma}\right](b)$.
Looking to $\mathcal{I}(\hat{y})$ as a function
of $\varepsilon_1,\varepsilon_2$, we define
$\hat{I}(\varepsilon_1,\varepsilon_2)=\mathcal{I}(\hat{y})-\xi$.
Thus, $\hat{I}(0,0)=0$.
On the other hand, applying integration by parts formulas
(Theorems~\ref{thm:gfip:Kop}, \ref{thm:IPL1} and \ref{thm:gfip}),
we obtain that
\begin{multline*}
\left.\frac{\partial\hat{I}}{\partial\varepsilon_2}\right|_{(0,0)}
=\int\limits_a^b\eta_2(t)\Biggl(k_\alpha(b,t)
\partial_2 G \left\{y\right\}_{P_2, P_3}^{\beta,\gamma}(t)
-\frac{d}{dt}\left(\partial_3G\left\{y\right\}_{P_2, P_3}^{\beta,\gamma}(t)
k_\alpha(b,t)\right)\\
-A_{P_2^*}^\beta\left[\tau \mapsto k_\alpha(b,\tau)
\partial_4 G\left\{y\right\}_{P_2, P_3}^{\beta,\gamma}(\tau)\right](t)
+K_{P_3^*}^\gamma\left[\tau \mapsto k_\alpha(b,\tau)
\partial_5 G\left\{y\right\}_{P_2, P_3}^{\beta,\gamma}(\tau)\right](t)\Biggr)dt,
\end{multline*}
where $P_j^*=<a,t,b,q_j,p_j>$, $j=1,2$.
We assume that $y$ is not an extremal for $\mathcal{I}$.
Hence, the fundamental lemma of the calculus of variations implies
that there exists a function $\eta_2$ such that
$\left.\frac{\partial\hat{I}}{\partial\varepsilon_2}\right|_{(0,0)}\neq 0$.
According to the implicit function theorem, there exists
a function $\varepsilon_2(\cdot)$ defined in a neighborhood of $0$ such that
$\hat{I}(\varepsilon_1,\varepsilon_2(\varepsilon_1))=0$.
Let $\hat{J}(\varepsilon_1,\varepsilon_2)=\mathcal{J}(\hat{y})$.
Function $\hat{J}$ has an extremum at $(0,0)$ subject to $\hat{I}(0,0)=0$,
and we have proved that $\nabla\hat{I}(0,0)\neq 0$. The Lagrange multiplier rule
asserts that there exists a real number $\lambda$ such that
$\nabla(\hat{J}(0,0)-\lambda\hat{I}(0,0))=0$.
Because
\begin{multline*}
\left.\frac{\partial\hat{J}}{\partial\varepsilon_1}\right|_{(0,0)}
=\int\limits_a^b\Biggl(k_\alpha(b,t)
\partial_2 F \left\{y\right\}_{P_2, P_3}^{\beta,\gamma}(t)
-\frac{d}{dt}\left(\partial_3 F\left\{y\right\}_{P_2,
P_3}^{\beta,\gamma}(t)k_\alpha(b,t)\right)\\
-A_{P_2^*}^\beta\left[\tau \mapsto k_\alpha(b,\tau)
\partial_4 F\left\{y\right\}_{P_2, P_3}^{\beta,\gamma}(\tau)\right](t)
+K_{P_3^*}^\gamma\left[\tau \mapsto k_\alpha(b,\tau)
\partial_5 F\left\{y\right\}_{P_2, P_3}^{\beta,
\gamma}(\tau)\right](t)\Biggr) \eta_1(t)dt
\end{multline*}
and
\begin{multline*}
\left.\frac{\partial\hat{I}}{\partial\varepsilon_1}\right|_{(0,0)}
=\int\limits_a^b\Biggl(k_\alpha(b,t)\partial_2
G \left\{y\right\}_{P_2, P_3}^{\beta,\gamma}(t)
-\frac{d}{dt}\left(\partial_3 G\left\{y\right\}_{P_2,
P_3}^{\beta,\gamma}(t)k_\alpha(b,t)\right)\\
-A_{P_2^*}^\beta\left[\tau \mapsto k_\alpha(b,\tau)\partial_4
G\left\{y\right\}_{P_2, P_3}^{\beta,\gamma}(\tau)\right](t)
+K_{P_3^*}^\gamma\left[\tau \mapsto k_\alpha(b,\tau)\partial_5
G\left\{y\right\}_{P_2, P_3}^{\beta,\gamma}(\tau)\right](t)\Biggr)\eta_1(t) dt,
\end{multline*}
one has
\begin{equation*}
\begin{split}
\int\limits_a^b & \Biggl\{k_\alpha(b,t)\partial_2
F \left\{y\right\}_{P_2, P_3}^{\beta,\gamma}(t)
-\frac{d}{dt}\left(\partial_3 F\left\{y\right\}_{P_2,
P_3}^{\beta,\gamma}(t)k_\alpha(b,t)\right)\\
&-A_{P_2^*}^\beta\left[\tau \mapsto k_\alpha(b,\tau)\partial_4
F\left\{y\right\}_{P_2, P_3}^{\beta,\gamma}(\tau)\right](t)
+K_{P_3^*}^\gamma\left[\tau \mapsto k_\alpha(b,\tau)\partial_5
F\left\{y\right\}_{P_2, P_3}^{\beta,\gamma}(\tau)\right](t)\\
&-\lambda\biggl(k_\alpha(b,t)\partial_2
G \left\{y\right\}_{P_2, P_3}^{\beta,\gamma}(t)
-\frac{d}{dt}\left(\partial_3 G\left\{y\right\}_{P_2,
P_3}^{\beta,\gamma}(t) k_\alpha(b,t)\right)\\
&-A_{P_2^*}^\beta\left[\tau \mapsto k_\alpha(b,\tau)\partial_4
G\left\{y\right\}_{P_2, P_3}^{\beta,\gamma}(\tau)\right](t)
+K_{P_3^*}^\gamma\left[\tau \mapsto k_\alpha(b,\tau)\partial_5
G\left\{y\right\}_{P_2, P_3}^{\beta,
\gamma}(\tau)\right](t)\biggr)\Biggr\}\eta_1 (t) dt=0.
\end{split}
\end{equation*}
From the fundamental lemma of the calculus of variations
(see, \textrm{e.g.}, \cite[Section~2.2]{G:H}) it follows
\begin{equation*}
\begin{split}
k_\alpha(b,t)&\partial_2 F\left\{y\right\}_{P_2, P_3}^{\beta,\gamma}(t)
-\frac{d}{dt}\left(\partial_3 F\left\{y\right\}_{P_2,
P_3}^{\beta,\gamma}(t)k_\alpha(b,t)\right)\\
&-A_{P_2^*}^\beta\left[\tau \mapsto k_\alpha(b,\tau)\partial_4
F\left\{y\right\}_{P_2, P_3}^{\beta,\gamma}(\tau)\right](t)
+K_{P_3^*}^\gamma\left[\tau \mapsto k_\alpha(b,\tau)\partial_5
F\left\{y\right\}_{P_2, P_3}^{\beta,\gamma}(\tau)\right](t)\\
&-\lambda\Biggl(k_\alpha(b,t)\partial_2 G
\left\{y\right\}_{P_2, P_3}^{\beta,\gamma}(t)
-\frac{d}{dt}\left(\partial_3 G\left\{y\right\}_{P_2,
P_3}^{\beta,\gamma}(t)k_\alpha(b,t)\right)\\
&-A_{P_2^*}^\beta\left[\tau \mapsto k_\alpha(b,\tau)\partial_4
G\left\{y\right\}_{P_2, P_3}^{\beta,\gamma}(\tau)\right](t)
+K_{P_3^*}^\gamma\left[\tau \mapsto k_\alpha(b,\tau)\partial_5
G\left\{y\right\}_{P_2, P_3}^{\beta,\gamma}(\tau)\right](t)\Biggr)=0,
\end{split}
\end{equation*}
that is,
\begin{multline*}
k_\alpha(b,t)\partial_2 H \left\{y\right\}_{P_2, P_3}^{\beta,\gamma}(t)
-\frac{d}{dt}\left(\partial_3 H\left\{y\right\}_{P_2,
P_3}^{\beta,\gamma}(t)k_\alpha(b,t)\right)\\
-A_{P_2^*}^\beta\left[\tau \mapsto k_\alpha(b,\tau)\partial_4
H\left\{y\right\}_{P_2, P_3}^{\beta,\gamma}(\tau)\right](t)
+K_{P_3^*}^\gamma\left[\tau \mapsto k_\alpha(b,\tau)\partial_5
H\left\{y\right\}_{P_2, P_3}^{\beta,\gamma}(\tau)\right](t)=0
\end{multline*}
with $H=F-\lambda G$.
\end{proof}

\begin{corollary}
Let $y$ be a minimizer to the isoperimetric problem
\begin{gather}
\mathcal{J}(y)={_{a}}\textsl{I}_b^\alpha\left[t \mapsto
F\left(t,y(t),{_{a}^{C}}\textsl{D}_t^\beta[y](t)\right)
\right](b) \longrightarrow \min ,\label{eq:21}\\
\mathcal{I}(y)={_{a}}\textsl{I}_b^\alpha\left[t \mapsto G\left(t,y(t),
{_{a}^{C}}\textsl{D}_t^\beta [y](t)\right)\right](b)=\xi,\label{eq:22}\\
y(a)=y_a,\quad y(b)=y_b.\label{eq:23}
\end{gather}
If $y$ is not an extremal of $\mathcal{I}$, then there exists
a constant $\lambda$ such that $y$ satisfies
\begin{equation}
\label{eq:20}
(b-t)^{\alpha-1}\partial_2 H\left(t,y(t),
{_{a}^{C}}\textsl{D}_t^\alpha[y](t)\right)
+{_{t}}\textsl{D}_b^\beta\left[\tau \mapsto (b-\tau)^{\alpha-1}\partial_3
H\left(\tau,y(\tau),{_{a}^{C}}\textsl{D}_\tau^\beta[y](\tau)\right)\right](t)=0
\end{equation}
for all $t\in(a,b)$, where $H(t,y,v)=F(t,y,v)-\lambda G(t,y,v)$.
\end{corollary}

\begin{proof}
Let $k_{\alpha}(x,t)=\frac{1}{\Gamma(\alpha)}(x-t)^{\alpha-1}$,
$h_{1-\beta}(t,\tau)=\frac{1}{\Gamma(1-\beta)}(t-\tau)^{-\beta}$,
$P_1=<a,b,b,1,0>$ and $P_2=<a,t,b,1,0>$. Then the $K$-op and the $B$-op
reduce to the left fractional integral and the left fractional
Caputo derivative, respectively. Therefore, problem \eqref{eq:21}--\eqref{eq:23}
is a particular case of problem \eqref{eq:IsoBound}--\eqref{eq:IsoFunct},
and \eqref{eq:20} follows from \eqref{eq:eqEL2}
with $\partial_3 H=\partial_5 H=0$.
\end{proof}

\begin{corollary}
\label{IsoPro:RicDel}
Let $y$ be a minimizer to
\begin{gather*}
\mathcal{J}(y)=\int_a^b F\left(t,y(t),y'(t),\,
B_{P_2}^\beta[y](t), K_{P_3}^\gamma[y](t) \right) dt
\longrightarrow \min, \\
\mathcal{I}(y)=\int_a^b G\left(t,y(t),y'(t),\,
B_{P_2}^\beta[y](t), K_{P_3}^\gamma[y](t) \right) dt=\xi,\\
y(a)=y_a,\ y(b)=y_b.
\end{gather*}
If $y$ is not an extremal of $\mathcal{I}$,
then there exists a constant $\lambda$ such that $y$ satisfies
\begin{equation}
\label{eq:CorEL}
\partial_2 H \left\{y\right\}_{P_2, P_3}^{\beta,\gamma}(t)
-\frac{d}{dt}\partial_3 H\left\{y\right\}_{P_2, P_3}^{\beta,\gamma}(t)
-A_{P_2^*}^\beta\left[\partial_4 H\left\{y\right\}_{P_2, P_3}^{\beta,\gamma}\right](t)
+K_{P_3^*}^\gamma\left[\partial_5 H\left\{y\right\}_{P_2, P_3}^{\beta,\gamma}\right](t)=0
\end{equation}
for all $t\in[a,b]$, where $H(t,y,u,v,w)=F(t,y,u,v,w)-\lambda G(t,y,u,v,w)$.
\end{corollary}

\begin{proof}
Let in problem \eqref{eq:IsoBound}--\eqref{eq:IsoFunct}
$P_1=<a,b,b,1,0>$ and kernel $k_{\alpha}(x,t) \equiv 1$. Then,
the generalized fractional integral $K_{P_1}^\alpha$
becomes the classical integral and \eqref{eq:CorEL}
follows from \eqref{eq:eqEL2}.
\end{proof}


\section{Illustrative examples}
\label{sec:ex}

We illustrate our results through two examples
with different kernels: one of a fundamental problem
\eqref{eq:1}--\eqref{eq:2} (Example~\ref{ex:1}),
the other an isoperimetric problem
\eqref{eq:IsoBound}--\eqref{eq:IsoFunct} (Example~\ref{ex:2}).

\begin{example}
\label{ex:1}
Let $\alpha,\beta \in\left(0,1\right)$, $\xi\in\mathbb{R}$,
$P_1=<0,1,1,1,0>$, and $P_2=<0,t,1,1,0>$.
Consider the following problem:
\begin{equation*}
\begin{gathered}
\mathcal{J}(y)=
K_{P_1}^\alpha\left[t \mapsto tK_{P_2}^\beta [y](t)
+\sqrt{1-\left( K_{P_2}^\beta[y](t) \right)^2}\right](1)
\longrightarrow \min,\\
y(0)=1 \, , \ y(1)=\frac{\sqrt{2}}{4}+\int_0^1 r_{\beta}(1-\tau)
\frac{1}{\left(1+\tau^2\right)^\frac{3}{2}}d\tau,
\end{gathered}
\end{equation*}
with kernel $h_{\beta}$ such that $h_\beta(t,\tau)=h_\beta(t-\tau)$ and $h_\beta(0)=1$.
Here the resolvent $r_{\beta}(t)$ is related to the kernel $h_{\beta}(t)$
by $r_{\beta}(t)=\mathcal{L}^{-1}\left[s \mapsto \frac{1}{s\widetilde{h}_{\beta}(s)}-1\right](t)$,
$\widetilde{h}_\beta(s)=\mathcal{L}\left[t \mapsto h_\beta(t)\right](s)$, where $\mathcal{L}$
and $\mathcal{L}^{-1}$ are the direct and the inverse Laplace operators, respectively.
We apply Theorem~\ref{theorem:ELCaputo} with Lagrangian $F$ given by
$F(t,y,u,v,w) =tw+\sqrt{1-w^2}$. Because
\begin{equation*}
y(t) = \frac{1}{\left(1+t^2\right)^\frac{3}{2}}
+\int_0^t r_\beta(t-\tau)\frac{1}{\left(1+\tau^2\right)^\frac{3}{2}}d\tau
\end{equation*}
is the solution to the Volterra integral equation of first kind
(see, \textrm{e.g.}, Equation~16, p.~114 of \cite{book:Polyanin})
\begin{equation*}
\textsl{K}_{P_2}^\beta [y](t)=\frac{t\sqrt{1+t^2}}{1+t^2},
\end{equation*}
it satisfies our generalized Euler--Lagrange equation
\eqref{eq:eqELCaputo}, \textrm{i.e.},
\begin{equation*}
\textsl{K}_{P_2^*}^\beta\left[\tau \mapsto k_\alpha(b,\tau)\left(
\frac{-\textsl{K}_{P_2^*}^\beta[y](\tau)}{\sqrt{1
-\left(\textsl{K}_{P_2^*}^\beta[y](\tau)\right)^2}}
+\tau\right)\right](t)=0.
\end{equation*}
In particular, for the kernel
$h_{\beta}(t-\tau)=\cosh(\beta(t-\tau))$, the boundary conditions are
$y(0)=1$ and $y(1)=1+\beta^2(1-\sqrt{2})$, and the solution is
$y(t)=\frac{1}{(1+t^2)^\frac{3}{2}}+\beta^2\left(1-\sqrt{1+t^2}\right)$
(\textrm{cf.} \cite[p.~22]{book:Polyanin}).
\end{example}

In the next example we make use of the Mittag--Leffler function
of two parameters: if $\alpha, \beta>0$, then
the Mittag--Leffler function is defined by
\begin{equation*}
E_{\alpha,\beta}(z)
=\sum_{k=0}^\infty\frac{z^k}{\Gamma(\alpha k+\beta)}\, .
\end{equation*}
This function appears naturally in the solution
of fractional differential equations,
as a generalization of the exponential function \cite{book:Kilbas}.

\begin{example}
\label{ex:2}
Let $\alpha,\beta\in\left(0,1\right)$, $\xi \in\mathbb{R}$,
and $\xi\notin\left\{\pm\frac{1}{4}\right\}$.
Consider the following problem:
\begin{equation}
\label{eq:ex}
\begin{gathered}
\mathcal{J}(y)={_{0}}\textsl{I}_1^\alpha\left[\sqrt{1+\left(y'
+\, ^{C}_{0}\textsl{D}_t^\beta [y]\right)^2}\right](1) \longrightarrow \min,\\
\mathcal{I}(y)={_{0}}\textsl{I}_1^\alpha\left[\left(y'
+ \, {^{C}_{0}}\textsl{D}_t^\beta [y]\right)^2\right](1) = \xi,\\
y(0)=0 \, , \ y(1)=\int_0^1 E_{1-\beta,1}\left(-(1-\tau)^{1-\beta}\right)
\frac{\sqrt{1-16\xi^2}}{4\xi}d\tau,
\end{gathered}
\end{equation}
which is an example of \eqref{eq:IsoBound}--\eqref{eq:IsoFunct} with
$p$-sets $P_1=<0,1,1,1,0>$ and $P_2=<0,t,1,1,0>$ and
kernels $k_\alpha(x-t)=\frac{1}{\Gamma(\alpha)}(x-t)^{\alpha-1}$
and $h_{1-\beta}(t-\tau)=\frac{1}{\Gamma(1-\beta)}(t-\tau)^{-\beta}$.
Function $H$ of Theorem~\ref{theorem:EL2} is given by
$H(t,y,u,v,w)=\sqrt{1+(u+v)^2} -\lambda (u+v)^2$.
One can easily check (see \cite[p.~324]{book:Kilbas}) that
\begin{equation}
\label{eq:y:ex}
y(t)=\int_0^t E_{1-\beta,1}\left(-(t-\tau)^{1-\beta}\right)
\frac{\sqrt{1-16\xi^2}}{4\xi}d\tau
\end{equation}
\begin{itemize}
\item is not an extremal for $\mathcal{I}$;
\item satisfies $y'+\,^{C}_{0}\textsl{D}_t^\beta [y]= \frac{\sqrt{1-16\xi^2}}{4\xi}$.
\end{itemize}
Moreover, \eqref{eq:y:ex} satisfies \eqref{eq:eqEL2} for $\lambda=2\xi$, \textrm{i.e.},
\begin{multline*}
-\frac{d}{dt}\left((1-t)^{\alpha-1}\left(y'(t)
+\,{^{C}_{0}}\textsl{D}_t^\beta[y](t)\right)\left(\frac{1}{\sqrt{1+\left(y'(t)
+\,{^{C}_{0}}\textsl{D}_t^\beta [y](t)\right)^2}}-4\xi\right)\right)\\
+ \, {_{t}}\textsl{D}_1^\beta\left[\tau \mapsto (1-\tau)^{\alpha-1}\left(y'(\tau)
+\,{^{C}_{0}}\textsl{D}_\tau^\beta [y](\tau)\right)\left(\frac{1}{\sqrt{1+\left(y'(\tau)
+\,{^{C}_{0}}\textsl{D}_\tau^\beta[y](\tau)\right)^2}}-4\xi\right)\right](t)=0
\end{multline*}
for all $t \in (0,1)$. We conclude that \eqref{eq:y:ex} is an extremal for problem \eqref{eq:ex}.
\end{example}


\section{Applications to Physics}
\label{sec:appl:phys}

If the functional \eqref{eq:1} does not depend on $B$-op and $K$-op,
then Theorem~\ref{theorem:ELCaputo} gives the following result:
if $y$ is a solution to the problem of extremizing
\begin{equation}
\label{FALVA:J}
\mathcal{J}(y)=\int_a^b F\left(t,y(t),y'(t)\right) k_\alpha(b,t) dt
\end{equation}
subject to $y(a)=y_a$ and $y(b)=y_b$, where $\alpha\in(0,1)$, then
\begin{equation}
\label{eq:FALVA}
\partial_2 F\left(t,y(t),y'(t)\right)-\frac{d}{dt}\partial_3
F\left(t,y(t),y'(t)\right)=\frac{1}{k_{\alpha}(b,t)}
\cdot \frac{d}{dt}k_{\alpha}(b,t)\partial_3 F\left(t,y(t),y'(t)\right).
\end{equation}
We recognize on the right hand side of \eqref{eq:FALVA}
the generalized weak dissipative parameter
\begin{equation*}
\delta(t)=\frac{1}{k_{\alpha}(b,t)}\cdot \frac{d}{dt}k_{\alpha}(b,t).
\end{equation*}


\subsection{Quantum mechanics of the damped harmonic oscillator}
\label{sub:sec:harosc}

As a first application, let us consider kernel
$k_\alpha(b,t)=\mathrm{e}^{\alpha (b-t)}$ and the Lagrangian
\begin{equation*}
L\left(y,\dot{y}\right)=\frac{1}{2}m \dot{y}^2-V(y),
\end{equation*}
where $V(y)$ is the potential energy and $m$ stands for mass.
The Euler--Lagrange equation \eqref{eq:FALVA} gives the following
second order ordinary differential equation:
\begin{equation}
\label{mod}
\ddot{y}(t)-\alpha \dot{y}(t) = -\frac{1}{m}V'(y(t)).
\end{equation}
Equation \eqref{mod} coincides with (14) of \cite{Herrera},
obtained by modification of Hamilton's principle.


\subsection{Fractional Action-Like Variational Approach (FALVA)}
\label{sub:sec:FALVA}

We now extend some of the recent results of \cite{Nabulsi2,Nabulsi3,Nabulsi4,Nabulsi},
where the fractional action-like variational approach (FALVA)
was proposed to model dynamical systems. FALVA functionals
are particular cases of \eqref{FALVA:J}, where the fractional time integral
introduces only one parameter $\alpha$. Let us consider the
Caldirola--Kanai Lagrangian \cite{Caldirola,Nabulsi3,Nabulsi}
\begin{equation}
\label{eq:CKLagr}
L\left(t, y, \dot{y}\right)
= m(t)\left(\frac{\dot{y}^2}{2}-\omega^2\frac{y^2}{2}\right),
\end{equation}
which describes a dynamical oscillatory system with exponentially
increasing time dependent mass, where $\omega$ is the frequency and
$m(t)=m_0 \mathrm{e}^{-\gamma b}\mathrm{e}^{\gamma t} = \bar{m}_0\mathrm{e}^{\gamma t}$,
$\bar{m}_0=m_0 \mathrm{e}^{-\gamma b}$. Using our generalized FALVA Euler--Lagrange
equation \eqref{eq:FALVA} with kernel $k_\alpha(b,t)$
to Lagrangian \eqref{eq:CKLagr}, we obtain
\begin{equation}
\label{eq:FALVAEL}
\ddot{y}(t)+\left(\delta(t)+\gamma\right)\dot{y}(t)+\omega^2y(t)=0.
\end{equation}
We study two particular kernels.
\begin{enumerate}
\item If we choose kernel
\begin{equation}
\label{eq:KuatEq}
k_\alpha (b,t)=\frac{(\rho+1)^{1-\alpha}}{\Gamma(\alpha)}\left(b^{\rho+1}
-t^{\rho+1}\right)t^\rho,
\end{equation}
defined in \cite{Katugampola}, then the Euler--Lagrange equation is
\begin{equation}
\label{eq:ELKuatugampola}
\partial_2 F\left(t,y(t),y'(t)\right)
-\frac{d}{dt}\partial_3 F\left(t,y(t),y'(t)\right)
=\left(\frac{(1-\alpha)(\rho+1)t^\rho}{b^{\rho+1}
-t^{\rho+1}}\right)\partial_3 F\left(t,y(t),y'(t)\right).
\end{equation}
In particular, when $\rho\rightarrow 0$, \eqref{eq:KuatEq}
becomes the kernel of the Riemann--Liouville fractional integral,
and equation \eqref{eq:ELKuatugampola} gives
\begin{equation*}
\partial_2 F\left(t,y(t),y'(t)\right)-\frac{d}{dt}\partial_3
F\left(t,y(t),y'(t)\right)
=\frac{1-\alpha}{b-t}\partial_3 F\left(t,y(t),y'(t)\right),
\end{equation*}
which is the Euler--Lagrange equation proved in \cite{Nabulsi3}.
For $\rho\neq 0$, we have
\begin{equation*}
\delta(t)=\frac{(1-\alpha)(\rho+1)t^\rho}{b^{\rho+1}-t^{\rho+1}}
\rightarrow 0 \text{ if } t\rightarrow\infty \text{ or } t\rightarrow 0.
\end{equation*}
Therefore, both at the very early time
and at very large time, dissipation disappears.
Moreover, if $\rho\rightarrow 0$, then
\begin{equation*}
\delta(t)=\frac{1-\alpha}{b-t}
\rightarrow
\begin{cases}
0 & \text{ if } t\rightarrow\infty \\
\frac{1-\alpha}{b} & \text{ if } t\rightarrow 0.
\end{cases}
\end{equation*}
This shows that at the origin of time, the time-dependent dissipation
becomes stationary, and that at very large time no dissipation, of any kind, exists.
\item If we choose kernel $k_{\alpha}(b,t)=\left(\cosh b- \cosh t\right)^{\alpha-1}$, then
\begin{equation}
\label{eq:extended}
\partial_2 F\left(t,y(t),y'(t)\right)
-\frac{d}{dt}\partial_3 F\left(t,y(t),y'(t)\right)
= -(\alpha-1)\frac{\sinh t}{\cosh b - \cosh t}\partial_3 F\left(t,y(t),y'(t)\right)
\end{equation}
and
\begin{equation*}
\delta(t)=-(\alpha-1)\frac{\sinh t}{\cosh b - \cosh t}
\rightarrow
\begin{cases}
\alpha-1 & \text{ if } t\rightarrow\infty\\
0 & \text{ if } t\rightarrow 0.
\end{cases}
\end{equation*}
In contrast with previous case, item 1, here
dissipation does not disappear at late-time dynamics.
\end{enumerate}

We note that there is a small inconsistence in \cite{Nabulsi3},
regarding to the coefficient of $\dot{y}(t)$ in \eqref{eq:FALVAEL},
and a small inconsistence in \cite{Nabulsi}, regarding a sign
of \eqref{eq:extended}.


\section{Conclusion}
\label{sec:conc}

In this article we unify, subsume and significantly extend
the necessary optimality conditions available in the literature
of the fractional calculus of variations.
It should be mentioned, however, that since fractional operators are nonlocal,
it can be extremely challenging to find analytical
solutions to fractional problems of the calculus of variations and,
in many cases, solutions may not exist.
In our paper we give two examples with analytic solutions, and many more can be found
borrowing different kernels from the book \cite{book:Polyanin}.
On the other hand, one can easily choose examples for which the
fractional Euler--Lagrange differential equations
are hard to solve, and in that case one needs to use numerical methods
\cite{agrawal:et:al:2012,Shakoor:01,MyID:221,MyID:225}.
The question of existence of solutions to fractional variational problems
is a complete open area of research. This needs attention.
Indeed, in the absence of existence, the necessary conditions
for extremality are vacuous: one cannot characterize an entity that does not exist in the
first place. For solving a problem of the fractional calculus of variations
one should proceed along the following three steps:
(i) first, prove that a solution to the problem exists;
(ii) second, verify the applicability of necessary optimality conditions;
(iii) finally, apply the necessary conditions which
identify the extremals (the candidates). Further elimination,
if necessary, identifies the minimizer(s) of the problem.
All three steps in the above procedure are crucial.
As mentioned by Young in \cite{young},
the calculus of variations has born from the study
of necessary optimality conditions, but any such theory is ``naive'' until
the existence of minimizers is verified. The process leading
to the existence theorems was introduced by
Leonida Tonelli in 1915 by the so-called direct method \cite{tonelli}.
During two centuries, mathematicians
were developing ``the naive approach to the calculus of variations''.
There was, of course, good reasons why the existence problem
was only solved in the beginning of XX century,
two hundred years after necessary optimality conditions began to be studied:
see \cite{cesari,torres2004} and references therein.
Similar situation happens now with the fractional
calculus of variations: the subject is only fifteen years old,
and is still in the ``naive period''. We believe
time has come to address the existence question,
and this will be considered in a forthcoming paper.


\section*{Acknowledgements}

This work was supported by {\it FEDER} funds through
{\it COMPETE} --- Operational Programme Factors of Competitiveness
(``Programa Operacional Factores de Competitividade'')
and by Portuguese funds through the
{\it Center for Research and Development
in Mathematics and Applications} (University of Aveiro)
and the Portuguese Foundation for Science and Technology
(``FCT --- Funda\c{c}\~{a}o para a Ci\^{e}ncia e a Tecnologia''),
within project PEst-C/MAT/UI4106/2011
with COMPETE number FCOMP-01-0124-FEDER-022690.
Odzijewicz was also supported by FCT through the Ph.D. fellowship
SFRH/BD/33865/2009; Malinowska by Bia{\l}ystok
University of Technology grant S/WI/02/2011;
and Torres by FCT through the Portugal--Austin (USA)
cooperation project UTAustin/MAT/0057/2008.



\end{document}